\documentclass[12pt]{amsart}

\usepackage{amssymb}
\usepackage{amscd}
\usepackage{enumerate}
\usepackage{graphicx}
\usepackage{color}

\numberwithin{equation}{section}

\theoremstyle{plain}

\newtheorem{theorem}{Theorem}[section]

\newtheorem{lemma}[theorem]{Lemma}

\newcommand\R{{I\!\!R}}

\DeclareFontFamily{OT1}{rsfs}{}
\DeclareFontShape{OT1}{rsfs}{n}{it}{<-> rsfs10}{}
\DeclareMathAlphabet{\mathscr}{OT1}{rsfs}{n}{it}

\begin{document}

\title{On the role of $L^3$ and $H^{\frac{1}{2}}$ norms in hydrodynamics}

\author{Laurent SCHOEFFEL}
\address{CEA Saclay, Irfu/SPP, Gif-sur-Yvette, France}
\email{laurent.olivier.schoeffel@cern.ch}

\begin{abstract}  

In this paper, we extend some results proved in previous references for three-dimensional Navier-Stokes equations. 
We show that
when the norm of the velocity field is small enough in $L^3({I\!\!R}^3)$, 
then a global smooth solution of the Navier-Stokes equations is ensured.
We show that a similar result holds when the norm of the velocity field is small enough in $H^{\frac{1}{2}}({I\!\!R}^3)$. 
The scale invariance of these two norms is discussed.

\end{abstract}

\maketitle

\section{Introduction}

The incompressible Navier-Stokes equations on the Euclidean space $\R^3$ can be expressed as
\begin{equation}
	\frac{\partial}{\partial t} \vec{v} + \left(\vec{v}\cdot\nabla\right)\vec{v}  = -\frac{1}{\rho_0} \nabla P +\frac{\eta}{\rho_0} \Delta \vec{v}.
\label{ns}
\end{equation}
where $\vec{v}$ is a divergence free vector field, 
with three components $(v_x, v_y, v_z)$ each of which may be functions of space and time coordinates $(x,y,z,t)$.
The three partial differential equations (\ref{ns}) represent
the mathematical description of the state of a moving fluid by relating its velocity field $\vec{v}(x,y,z,t)$
and any two thermodynamic quantities pertaining to the fluid, for example the pressure 
$P(x,y,z,t)$and the mass density $\rho(x,y,z,t)$. The mass density is  constant $\rho_0$  for an
incompressible fluid.
The positive coefficient $\eta$ is called the dynamic viscosity, 
describing the quality of the fluid, while $\nu=\frac{\eta}{\rho_0}$ is called the kinematic viscosity.
In general the viscosity coefficient is a function of pressure and temperature of the fluid. As pressure
and temperature may not be constant throughout the fluid, the viscosity coefficient also may not be
constant throughout the fluid. 
The incompressibility condition reads $\nabla\cdot \vec{v} = 0$. This last relation is equivalent to
$\rho(x,y,z,t)=\rho_0$. A general introduction to Navier-Stokes equations can be found in \cite{TOTO}.

The problem of whether the three-dimensional incompressible
Navier-Stokes equations (\ref{ns}) can develop a finite time singularity starting from
smooth initial conditions or if a (unique) global smooth vector field
solution exists, which means smooth for all times,
is still
unresolved 
\cite{temam:1984,temam:1988,constantin-foias:1988,existence,majda-bertozzi:2002}. 
By smooth functions, we mean functions that are
divergence free, infinitely differentiable, square integrable and with strong decay properties at infinity.
This problem is
recognized as one of the Millennium problems
\cite{ladyzhenskaya:2003,fefferman:2006}. 

The purpose of this paper is to prolong some partial results and arguments that we have introduced in \cite{TOTO}, on the role
of $L^p(\R^3)$ norms of the velocity field, higher than $L^2(\R^3)$ (definitions of these norms are recalled in the next section).
More precisely, our purpose is to show that if $\| \vec{v} \|^2_{L^3(\R^3)}$ is small enough, in a sense that will be defined in the following,
then a global smooth solution of the Navier-Stokes equations exists.

\section{Definitions and Notations}

Let us precise the definitions that we use extensively in the following. The Euclidean norm (squared) of a vector $\vec{v}$
is written as $\| \vec{v} \|^2=\sum_{i=1}^{3} |v_i|^2$, where the sum runs over indices $i=(1,2,3)$ or equivalently 
$(x,y,z)$.
The $L^2(\R^3)$ norm of a vector field $\vec{v}(x,y,z)$ is defined as
\begin{equation}
\| \vec{v} \|_{L^2(\R^3)} = \left[ \int_{\R^3} \|  \vec{v}(x,y,z) \|^2 dV \right]^{1/2},
\label{defL2}
\end{equation}
where the integrand is simply the Euclidean norm (squared) of the vector $\vec{v}(x,y,z)$.
In short, a function defined on the Euclidean space $\R^3$ possesses a $L^2(\R^3)$ norm if it is square integrable.
Obviously, if the vector field is also a function of $t$, $\vec{v}(x,y,z,t)$, then the $L^2(\R^3)$ norm (\ref{defL2})
becomes also a function of $t$.
Following equation (\ref{defL2}), the $L^2(\R^3)$ of the first gradient of $\vec{v}(x,y,z)$
is given by
$$
\| \nabla \vec{v} \|_{L^2(\R^3)} = \left[ \int_{\R^3} \| \nabla  \vec{v}(x,y,z) \|^2 dV \right]^{1/2}.
$$
In general, the $L^p(\R^3)$ norm of a vector field that may depend on $t$, $\vec{v}(x,y,z,t)$, is defined as
\begin{equation}
\| \vec{v} \|_{L^p(\R^3)}(t) = \left[ \int_{\R^3} \|  \vec{v}(x,y,z,t) \|^p dV \right]^{1/p}.
\label{defLp}
\end{equation} 
Then, a function $f$ defined on the Euclidean space $\R^3$ possesses a $L^{p > 2}(\R^3)$ norm if 
the integral of $|f|^p$ is defined. This extends the property of square integrable functions to any exponent $p$.
An
interesting property holds for $L^p(\Omega)$ norms defined on a subset of $\R^3$ of finite dimensions.
It can be shown that $L^q(\Omega)$ with $q>p$ is embedded in $L^p(\Omega)$, or equivalently there exists a constant $K$
independent of the subset $\Omega$ such that
$$
\| \vec{v} \|_{L^p(\Omega)}(t) \le K \| \vec{v} \|_{L^{q>p}(\Omega)}(t).
$$
This means that higher $L^{q>p}(\Omega)$ norms can control lower ones for bounded subset or $\R^3$.
This can also be formulated as the general property that higher $L^{q}(\Omega)$ norms have more regularity
than smaller ones.
In the following we show how to extend this kind of observation in the context of the
incompressible Navier-Stokes equations on the Euclidean space $\R^3$.

\section{Initial data with small $L^3$ norm }

Based on the Navier-Stokes equations and the previous notations, there are two important relations that we can derive
concerning the  $L^2(\R^3)$ of the velocity field and its first gradient. They read
\begin{eqnarray}
 \frac{1}{2} \frac{d}{dt}  \| \vec{v} \|^2_{L^2(\R^3)}  &=&- \nu \int_{\R^3} \| \nabla \vec{v} \|^2 dV \\
\label{ns2a}
\frac{1}{2} \frac{d}{dt}  \|  \nabla \vec{v} \|^2_{L^2(\R^3)} +
\nu \| \Delta \vec{v} \|^2_{L^2(\R^3)} &=&
 -\int_{\R^3} \partial_i v_j \partial_i v_k \partial_k v_j dV,
\label{ns2b}
\end{eqnarray}
where we omit the sum operator for repeated indices.
A complete proof of equations (\ref{ns2a}) and (\ref{ns2b}) can be found in \cite{TOTO}.
They are essential for addressing the Millennium problem in the following sense:
(i) equation (\ref{ns2a}) implies that $\| \vec{v} \|^2_{L^2(\R^3)}(t)$
is bounded up at all times once $\| \vec{v} \|^2_{L^2(\R^3)}(t=0)$ is finite;
(ii) for some cases, equation (\ref{ns2b}) implies that $ \|  \nabla \vec{v} \|^2_{L^2(\R^3)}(t)$
is bounded up at all times \cite{TOTO}.
Therefore, once this is ensured that the $L^2(\R^3)$ norms of the velocity field  and its first gradient 
(in some cases) are
bounded up at all times, then a global smooth solutions of the Navier-Stokes equations exists
\cite{temam:1984,temam:1988}.
In \cite{TOTO} we have discussed some particular cases for which the $L^2(\R^3)$ norm of the first gradient of
the velocity field is defined. All those cases depend obviously on how we can bound up
the integral $\int_{\R^3} \partial_i v_j \partial_i v_k \partial_k v_j dV$.
Let us remind the trivial case of the two-dimensional Euclidean space where this integral is zero. This proves immediately
that the first gradient of the velocity field possesses a $L^2(\R^2)$ norm and therefore, the existence of a global solution
to Navier-Stokes equations on $\R^2$. 

In this paper, we want to show that if $\| \vec{v} \|^2_{L^3(\R^3)}$ is small enough, in a sense that needs to be defined precisely,
then a global smooth solution of the Navier-Stokes equations exists.
First, we prove the following lemma.

\begin{lemma}
For solutions of Navier-Stokes equations, with smooth initial conditions,
there exists a positive constant $C$  such that
\begin{equation}
 \frac{1}{2} \frac{d}{dt}  \|  \nabla \vec{v} \|^2_{L^2(\R^3)} 
+
{\nu} \|  \Delta \vec{v} \|^{2}_{L^2(\R^3)}
\le
C \
 \|         \vec{v} \|_{L^3(\R^3)} 
 \|  \Delta \vec{v} \|^{2}_{L^2(\R^3)}
\label{dv20}
\end{equation}
\label{lemma2}
\end{lemma}
\begin{proof} 

To  prove this lemma, we need bound up the right hand side of equation (\ref{ns2b}).
First, by integration by parts and using the incompressibility condition $\partial_i v_i=0$, we can write
$$
\int_{\R^3} \partial_i v_j \partial_i v_k \partial_k v_j dV
=
\int_{\R^3}  v_j \partial_i^2 v_k \partial_k v_j dV.
$$
Then, 
we can bound each gradient by its Euclidean length and  apply the 
the Holder's identity. We obtain
\begin{equation}
-\int_{\R^3} \partial_i v_j \partial_i v_k \partial_k v_j dV
\le
\|         \vec{v} \|_{L^3(\R^3)} 
\|  \nabla \vec{v} \|_{L^6(\R^3)} 
\|  \Delta \vec{v} \|_{L^2(\R^3)}.
\label{in1}
\end{equation}
We need to deal with the term in $\|  \nabla \vec{v} \|_{L^6(\R^3)}$.
Ideally, we would like to express it in term of a smaller $L^p(\R^3)$ norm. This can be achieved by
using a general Sobolev's inequality on $\R^3$: 
$$
\| u \|_{L^{3p/(3-p)}(\R^3)} \le C(p) \ \|  \nabla {u} \|_{L^p(\R^3)},
$$
verified for a function $u$ defined on $\R^3$ such that the above integrals exist and
where $C(p)$ is a constant depending only on $p$. We apply this inequality for $p=2$ with $u$ being the
first gradient of the velocity field. We get
\begin{equation}
\|  \nabla \vec{v} \|_{L^6(\R^3)} 
\le 
C \
\|  \Delta \vec{v} \|_{L^2(\R^3)},
\label{in2}
\end{equation}
Where $C$ is a constant that does not depend on the velocity field.
Finally, combining equations (\ref{in1}) and (\ref{in2}), we obtain
$$
-\int_{\R^3} \partial_i v_j \partial_i v_k \partial_k v_j dV
\le
C
\|         \vec{v} \|_{L^3(\R^3)} 
\|  \Delta \vec{v} \|^2_{L^2(\R^3)},
$$
which completes the proof.
\end{proof}
Let us rewrite equation (\ref{dv20}) as
\begin{equation}
\frac{d}{dt}  \|  \nabla \vec{v} \|^2_{L^2(\R^3)} 
+
\left( 2{\nu}  - 2C \|         \vec{v} \|_{L^3(\R^3)} \right)
\|  \Delta \vec{v} \|^{2}_{L^2(\R^3)}
\le 
0.
\label{dv20b}
\end{equation}
The last step is to use the Poincar\'e's inequality
$$
\|  \Delta \vec{v} \|^{2}_{L^2(\R^3)} \ge K \|  \nabla \vec{v} \|^2_{L^2(\R^3)},
$$
where $K$ is a positive constant. Therefore, if the norm of the velocity field in $L^3(\R^3)$ verifies
the condition
$\|         \vec{v} \|_{L^3(\R^3)} \le {\nu}/C$, then
we obtain 
$$
\frac{d}{dt}  \|  \nabla \vec{v} \|^2_{L^2(\R^3)} 
+
K \left( 2{\nu}  - 2C \|         \vec{v} \|_{L^3(\R^3)} \right)
\|  \nabla \vec{v} \|^2_{L^2(\R^3)}
\le 
0.
$$
This proves that
$ \|  \nabla \vec{v} \|^2_{L^2(\R^3)}(t)$
exists at all times and consequently a global solution to Navier-Stokes equations exists also.

In summary, 
when the norm of the velocity field is small enough in $L^3(\R^3)$ at initial time, $\| \vec{v} \|^2_{L^3(\R^3)} \le \nu/C$,
then a global smooth solution of the Navier-Stokes equations exists.
In \cite{TOTO}, we have discussed the case where the $L^4(\R^3)$ norm of the velocity field is bounded up, 
from which the existence of global smooth solutions can be deduced.
The result of this paper is stronger in the sense that we have proved the result for a smaller norm.

\section{$L^3$ norm and scale invariance}

A key property of the Navier-Stokes equations is that they are invariant under the scale
transformation \cite{TOTO}
\begin{eqnarray}
\vec{v}_\lambda(\vec{x},t)&=&\lambda\vec{v}(\lambda\vec{x},\lambda^2t) \label{lambda1} \\
P_\lambda(\vec{x},t)&=&\lambda^2 P(\lambda\vec{x},\lambda^2t) .
\end{eqnarray}

This means that if the velocity field $\vec{v}(x,y,z,t)$ 
is a solution of the equations, then the velocity field $\vec{v}_\lambda(x,y,z,t)$ 
is another
acceptable solution (by construction).
We can think of this transformation, with $\lambda \gg 1$, as taking the fine scale behavior of the velocity field 
$\vec{v}(x,y,z,t)$and matching it with an identical (but rescaled and slowed down) coarse scale
behavior $\vec{v}_{1/\lambda}(x,y,z,t)$.

Interestingly, we can observe how different 
$L^p(\R^3)$ norms of the velocity field are transformed by this scale transformation
\begin{eqnarray}
 \| \vec{v}_{1/\lambda} \|^2_{L^2(\R^3)}&=& \lambda \| \vec{v} \|^2_{L^2(\R^3)} \\
 \| \vec{v}_{1/\lambda} \|^3_{L^3(\R^3)}&=&  \| \vec{v} \|^3_{L^3(\R^3)} \\
 \| \vec{v}_{1/\lambda} \|^4_{L^4(\R^3)}&=& \frac{1}{\lambda} \| \vec{v} \|^4_{L^4(\R^3)},
\end{eqnarray}
which can be continued trivially for any $L^p(\R^3)$ norms.
The important feature of the $L^3(\R^3)$ norm of the velocity field is that it is scale invariant.
Hence, a bound on the $L^3(\R^3)$ norm of the velocity field is invariant for
$\vec{v}_{1/\lambda}$ when the parameter $\lambda$ is increased.
We have shown in the previous section that if this norm is small enough, no blow up can occur
by shifting the solution of Navier-Stokes equations to smaller scales. This is consistent with the 
mathematical property of the scale invariance of this particular $L^3(\R^3)$  norm.

Let us conclude by remarking that the same property of scale invariance holds for 
the norm of the velocity field defined in the Sobolev's set $H^{\frac{1}{2}}$ .
The $H^{\frac{1}{2}}$ norm of a function defined on the real space $\R$ can be expressed as
$$
 \| u \|^2_{H^{\frac{1}{2}}} = \int_{\R} \left[ 1+|\xi|^2 \right]^{1/2} |\bar{u}(\xi)|^2 d\xi,
$$
where  $\bar{u}$ is the Fourier transform of the function $u$.
In $\R^3$ the generalization is trivial. 
As the scale invariance  holds for the $H^{\frac{1}{2}}(\R^3)$ norm of the velocity field, we also have the
property:
if the norm of the velocity field is small enough in $H^{\frac{1}{2}}(\R^3)$ at initial time, 
then a global smooth solution of the Navier-Stokes equations exists.


\end{document}